\documentclass[11pt]{amsart}
\usepackage{amsmath}
\newtheorem{thm}{{\bf Theorem}}[section]
\newtheorem{cor}[thm]{{\bf Corollary}}
\newtheorem{lem}[thm]{{\bf Lemma}}
\newtheorem{prop}[thm]{{\bf Proposition}}
\theoremstyle{definition}

\theoremstyle{remark}

\newtheorem{exe}[thm]{{\bf  Example}}
\numberwithin{equation}{section}

\begin{document}
\title{Regularity action of abelian linear groups on $\mathbb{C}^{n}$}

\author{Adlene Ayadi and Ezzeddine Salhi}

\address{Adlene Ayadi$^{1}$, \ Department of Mathematics, Faculty of Sciences of Gafsa, Tunisia;  Ezzeddine
Salhi$^{3}$,  Department of Mathematics, Faculty of Sciences of
Sfax \, B.P. 802, 3018 Sfax, Tunisia} \email{ adleneso@yahoo.fr;\
\ Ezzeddine.Salhi@fss.rnu.tn.}

\thanks{This work is supported by the research unit: syst\`emes dynamiques et combinatoire: 99UR15-15}

\subjclass[2000]{37C85}

\keywords{orbit,regular, dense orbit, minimal, subgroup,\dots}

\maketitle

\begin{abstract}

In this paper, we give a characterization of the action of any
abelian subgroup $G$ of $GL(n,\mathbb{C})$ on $\mathbb{C}^{n}$. We
prove that any orbit of $G$ is regular with order  $m\leq 2n$.
Moreover, we give a method to determine this order. In the other
hand, we specify the region of all orbits which are isomorphic. If
G is finitely generated, this characterization is explicit.
\end{abstract}

\section{{\bf Introduction} }

Let $GL(n, \mathbb{C})$ be the group of all reversible square
matrix  over $\mathbb{C}$ with order $n$ and let  $G$  be an abelian subgroup of
GL($n, \mathbb{C})$. There is a natural
linear action $GL(n, \mathbb{C})\times \mathbb{C}^{n}  : \
\longrightarrow\ \mathbb{C}^{n}$.  $(A, v) \ \longmapsto \ Av$.
For a vector $v \in \mathbb{C}^{n}$, denote by $G(v) = \{Av, \ A
\in G \}\subset \mathbb{C}^{n}$ the \emph{orbit} of $G$ through
$v$. A subset $E \subset \mathbb{C}^{n}$  is called
\emph{$G$-invariant} if $A(E)\subset E$   for any   $A\in G$; that
is   $E$   is a union of orbits. Denote by $\overline{E}$ (resp.
$\overset{\circ}{E}$ ) the closure (resp. interior) of $E$.

 An orbit $\gamma$ is
called {\it regular} with order $m$ if for every $v\in \gamma$
there exists an open set $O$ containing $v$ such that
$\overline{\gamma}\cap O$ is a manifold with dimension $m$ over
$\mathbb{R}$. In particular, $\gamma$ is  locally dense in
$\mathbb{C}^{n}$ if and only if  $m=2n$, and $\gamma$ is discrete
if and only if $m=0$. Notice that, the closure of a regular orbit
is not necessary a manifold (see example 8.4). We say that the action of $G$
is {\it regular} on $\mathbb{C}^{n}$ if every orbit of $G$ is regular. Here, the question
to investigate is the following:\
\\
\emph{(1) The orbits of $G$ are they regular?}\
\\
\emph{(2) If $G$ has a  regular orbit, how can we determine its
order?}
\medskip

The notion of regular orbit is a generalization of non exceptional
orbit defined for the action of any group of homeomorphism on a
topological space $X$. A nonempty compact subset $Y\subset X$ is a
minimal set if for every $y\in Y$ the orbit of $y$ is dense in
$Y$. In \cite{G}, Gottschalk discussed the question of what sets
can be minimal sets. A minimal set which is a Cantor set is called
an exceptional set. Their dynamics were recently initiated for
some classes in different point of view, (see for instance,
\cite{HP},\cite{Denj},\cite{YK},\cite{PS},\cite{VG},\cite{GK.KK}).

\medskip

For a subset $E\subset \mathbb{C}^{n}$, denote by $vect(E)$ the
vector subspace of $\mathbb{C}^{n}$ generated by all elements of
$E$. For every $u\in \mathbb{C}^{n}$ denote by;\
\\
- $E(u)=vect(G(u))$\
\\
- $r(u)=\mathrm{dim}(E(u))$  over $\mathbb{C}$.\
\\
- $rank(G(u))=\mathrm{dim}(E(u))$.\

 \bigskip

S.Chihi proved in \cite{S.Ch} that for an abelian linear action,
every orbit $\gamma$ is contained in a locally closed sub-manifold
$V(\gamma)$ such that $\gamma\subset V(\gamma)\subset
\overline{\gamma}$. Moreover, he shows that if
$\overline{\gamma_{1}}=\overline{\gamma_{2}}$ then $\gamma_{1}$
and $\gamma_{2}$ are isomorphic and
$rank(\gamma_{1})=rank(\gamma_{2})$.

\medskip

The purpose of this paper is to give a complete answer to the
above question for any abelian subgroup of $GL(n, \mathbb{C})$. In
\cite{aAhM05}, the authors present a global dynamic of every
abelian subgroup of $GL(n, \mathbb{C})$  and in \cite{aAh-M05},
they gave a characterization of existence of dense orbit for any
abelian subgroup of $GL(n, \mathbb{C})$. Our main result is viewed
as continuation of work in \cite{aAhM05} and \cite{aAh-M05}. We
found similar result given in
 \cite{S.Ch} as a consequence of Theorem ~\ref{T:3}, and we prove that every orbit is regular and we
characterize its order $m$. If $G$ is  finitely generated, this
characterization is explicit.\
\\
Denote by:
\\
- $\mathbb{C}^{*}=\mathbb{C}\backslash\{0\}$,
$\mathbb{R}^{*}=\mathbb{R}\backslash\{0\}$ and
$\mathbb{N}^{*}=\mathbb{N}\backslash\{0\}$.\
\
\\
- $e^{(k)} = [e^{(k)}_{1},\dots,  e^{(k)}_{r}]^{T}\in
\mathbb{C}^{n+1}$ where $$e^{(k)}_{j} = \left\{\begin{array}{c}
  0\in \mathbb{C}^{n_{j}}\ \ \mathrm{if}\ \ j\neq k \\
  e_{k,1}\ \ \ \ \ \ \ \ \mathrm{if}\ \ j = k  \\
\end{array}\right. \ \ \ \ \ \  for \ \ every\ \ 1\leq j,k\leq r.$$
\
\\
- $M_{n}(\mathbb{C})$  the set of all square matrix of order
$n\geq 1$ with coefficients in $\mathbb{C}$.\
\\
- $M_{p,q}(\mathbb{C})$ the set of all matrix having $p$ lines and
$q$ colons with coefficients in $\mathbb{C}$.\
\\
- $\mathbb{T}_{m}(\mathbb{C})$ the set of matrices over
$\mathbb{C}$ of the form $$\begin{bmatrix}
  \mu &  &   & 0 \\
  a_{2,1} & \mu &   &  \\
  \vdots &  \ddots & \ddots &  \\
  a_{m,1} & \dots & a_{m,m-1} & \mu
\end{bmatrix} \qquad (1)$$
\\
\\
-\; $\mathbb{T}_{m}^{\ast}(\mathbb{C})$  the group of matrices of
the form $(1)$ with $\mu\neq 0$.
\\
Let $r\in \mathbb{N}^{*}$ and
$\eta=(n_{1},\dots,n_{r})\in(\mathbb{N}^{*})^{r}$ such that
$\underset{i=1}{\overset{r}{\sum}}n_{i}=n$. Denote by:\
\\
 - $\mathcal{K}_{\eta,r}(\mathbb{C}) = \mathbb{T}_{n_{1}}(\mathbb{C})\oplus\dots\oplus\mathbb{T}_{n_{r}}(\mathbb{C}).$
\\
-
$\mathcal{K}^{*}_{\eta,r}(\mathbb{C})=\mathcal{K}_{\eta,r}(\mathbb{C})\cap
GL(n, \mathbb{C})$.\
\\
- $\mathcal{C}_{0} = (e_{1},\dots,e_{n})$  the canonical basis of
$\mathbb{C}^{n}$.

\bigskip

 The author have
proved in \cite{aAh-M05}, that for every abelian subgroup of
$GL(n, \mathbb{C})$ there exists $P\in GL(n, \mathbb{C})$ such
that $P^{-1}GP$ is a subgroup of
$\mathcal{K}^{*}_{\eta,r}(\mathbb{C})$ for some
$r\in\mathbb{N}^{*}$ and
$\eta=(n_{1},\dots,n_{r})\in(\mathbb{N}^{*})^{r}$ (see Proposition
~\ref{p:8}). We say that $\widetilde{G}=P^{-1}GP$ is  {\it a
normal form} of $G$.  We let\
\\
- $\mathrm{g}
=\exp^{-1}(G)\cap\left[P\left(\mathcal{K}_{\eta,r}(\mathbb{C})\right)P^{-1}\right]$\
\\
- $\mathrm{g}_{u} = \{Bu, \  \ B\in \mathrm{g} \}, \
u\in\mathbb{C}^{n}.$\
\\
 One has   $exp(\mathrm{g}) = G$
(see Lemma ~\ref{L:1}).

\bigskip

For any closed additive subgroup $F$ of $\mathbb{C}^{n}$, we say
that dim$(F)=s$, if $s$ is the bigger dimension of all vector
spaces over $\mathbb{R}$ contained in $F$. Moreover, if $F$ is
considered as a manifold over $\mathbb{R}$, then dim$(F)=s$. (See
Proposition ~\ref{p:3}).
\bigskip

 Finally, consider the following rank condition on a collection of
vectors $u_{1},\dots,u_{p}\in \mathbb{R}^{n}$, $p>n$. Suppose that
$(u_{1},\dots,u_{n})$ is a basis of $\mathbb{R}^{n}$ and there
exists $0\leq m\leq n$ such that
$u_{k}=\underset{j=1}{\overset{m}{\sum}}\alpha_{k,j}u_{n-m+j}$,
for every $n+1\leq k\leq p$, $\alpha_{k,j}\in\mathbb{R}^{*}$.
\
\\
\\
$\bullet$  We say that $u_{1},\dots,u_{p}\in \mathbb{R}^{n}$
satisfy \emph{property} $\mathcal{D}(m)$ if and only if  for every
$(t_{1},...,t_{m},s_{1},...,s_{p-n})\in
 \mathbb{Z}^{p-n+m}-\{0\}$ :
$$rank\left(\left[\begin{array}{cccccccc }
 1&0&\dots& 0&\alpha_{n+1,1} &\dots  &\dots  & \alpha_{p,1} \\
  0&\ddots & \ddots& \vdots&\vdots & \vdots & \vdots & \vdots \\
  \vdots&\ddots & \ddots&0 & \vdots & \vdots & \vdots& \vdots \\
 0& \dots & 0& 1& \alpha_{n+1,m} &\dots  &\dots  & \alpha_{p,m} \\
  t_{1}& \dots &\dots& t_{m}& s_{1} &\dots &\dots  & s_{p-n}
 \end{array}\right]\right) =\ m+1.$$
 and this is equivalent by Lemma~\ref{L:13} to say  that \ \ $\mathbb{Z}u_{n-m+1}+\dots+\mathbb{Z}u_{p}$
 is dense in $\mathbb{R}u_{n-m+1}\oplus\dots\oplus\mathbb{R}u_{n}$.
\
\\
\\
$\bullet$ For every permutation $\sigma\in\mathcal{S}_{p}$ we say
that $u_{\sigma(1)},\dots,u_{\sigma(p)}\in \mathbb{R}^{n}$ satisfy
\emph{property} $\mathcal{D}(m)$  if $u_{1},\dots,u_{p}\in
\mathbb{R}^{n}$ satisfy \emph{property} $\mathcal{D}(m)$.
\
\\
\\
For a vector $v\in\mathbb{C}^{n}$, we write $v = \textrm{Re}(v)+
i\textrm{Im}(v)$ where $\textrm{Re}(v), \ \textrm{Im}(v)\in
\mathbb{R}^{n}$.\
\medskip

Let  $\theta : \mathbb{C}^{n}\longrightarrow \mathbb{R}^{2n}$ be
the isomorphism, defined by
$$\theta(z_{1},\dots,z_{n})=(Re(z_{1}),\dots,Re(z_{n});
Im(z_{1}),\dots,Im(z_{n})).$$\
\\
$\bullet$ We say that $v_{1},\dots,v_{p}\in\mathbb{C}^{n}$ satisfy
\emph{property} $\mathcal{D}(m)$, if
$\theta(v_{1}),\dots,\theta(v_{p})\in\mathbb{R}^{2n}$  satisfy
\emph{property} $\mathcal{D}(m)$.

\bigskip
\
\\
\\
Our principal results can be stated as follows:

\begin{thm} \label{T:1}Let $G$ be an abelian subgroup of $GL(n, \mathbb{C})$. Then for every $u\in
\mathbb{C}^{n}$, there exists a $G$-invariant dense open subset
$U_{u}$ of $E(u)$, containing $u$, such that:
\begin{itemize}
  \item[(i)] For every $v\in U_{u}$, we have $E(v)=E(u)$.
  \item[(ii)] All orbit in
$U_{u}$ are isomorphic.
  \item[(iii)] $E(u)\backslash U_{u}$ \ is a union of at most $r(u)$ $G$-invariant  vector subspaces of $E(u)$ with
  dimension $r(u)-1$ over $\mathbb{C}$.
\end{itemize}
\end{thm}
\medskip

By the following Theorem, the order of any orbit is equal to
 dimension of a closed additive subgroup of $\mathbb{C}^{n}$, over $\mathbb{R}$.
\begin{thm}\label{T:2} Let  $G$  be an abelian subgroup of  GL$(n,\ \mathbb{C})$ and $u\in\mathbb{C}^{n}$.
The following are equivalent:
\begin{enumerate}
    \item [(i)]  $G(u)$  is regular with order $m$.
    \item [(ii)]  For every $v\in U_{u}$,  $G(v)$  is regular with order $m$.
    \item [(iii)]  dim$\left(\overline{\mathrm{g}_{u}}\right)=m$ over $\mathbb{R}$.
\end{enumerate}
\end{thm}
\medskip

As a consequence from Theorem ~\ref{T:2} we prove, by the following
corollary, that the action of any abelian subgroup of
$GL(n, \mathbb{K})$ on $\mathbb{K}^{n}$ is regular ($\mathbb{K}=\mathbb{R}$ or $\mathbb{C}$).

\begin{cor}\label{C:02} Let  $G$ be an abelian subgroup of $GL(n,
\mathbb{C}) ($resp. $GL(n, \mathbb{R}))$. Then every orbit of $G$
is regular with order $0\leq m \leq 2n\ ($resp. $0\leq m \leq n)$.
\end{cor}
\medskip

Where $GL(n, \mathbb{R})$ denotes the group of all reversible
square matrix  over $\mathbb{R}$ with order $n$.

\begin{cor}\label{C:2} Let $G$ be an abelian subgroup of $GL(n, \mathbb{K})$ $(\mathbb{K}=\mathbb{R}$ or $\mathbb{C})$.
\begin{itemize}

\item [(i)] If  $\mathbb{K}=\mathbb{R}$ then an orbit $G(u)$ is regular with order $m = n$   \ if and
only if  it  is locally dense.

 \item [(ii)] If  $\mathbb{K}=\mathbb{C}$ then an orbit $G(u)$ is regular with order $m = 2n$  \ if and
only if   \ it is dense in $\mathbb{C}^{n}$.
 \end{itemize}
\end{cor}
\medskip

\begin{thm}\label{T:3} Let  $G$  be an abelian subgroup of $GL(n,
\mathbb{C})$ and $u\in \mathbb{C}^{n}$. Then the closure
$\overline{G(u)}$  is a vector subspace of  $\mathbb{C}^{n}$ if
and only if  $G(u)$ is regular with order $2r(u)$.
\end{thm}
\bigskip

For a finitely generated subgroup \ $G\subset \textrm{GL}(n,
\mathbb{C})$, \ such that its normal form $\widetilde{G}=P^{-1}GP$
is a subgroup of $\mathcal{K}^{*}_{\eta,r}(\mathbb{C})$.  We give
an explicit condition to determine the order of any orbit.
\medskip

\begin{thm}\label{T:4} If $G$ is an abelian subgroup of GL$(n, \mathbb{C})$  generated by
$A_{1},\dots,A_{p}$ and let  $B_{1},\dots,B_{p}\in \mathrm{g}$
such that $A_{1} = e^{B_{1}},\dots, A_{p} = e^{B_{p}}$. Then for
every  $u\in \mathbb{C}^{n}\backslash\{0\}$. The following are
equivalent:
\begin{enumerate}
\item [(i)]  $G(u)$  is regular with order $m$.
\item [(ii)]  $B_{1}u,\dots,B_{p}u,\ 2i\pi P(e^{(1)}),\dots,2i\pi P(e^{(r)}) $  satisfy property
$\mathcal{D}(m)$.
\item [(iii)] $\mathrm{g}_{u} =
\underset{k=1}{\overset{p}{\sum}}\mathbb{Z}(B_{k}u) +
\underset{k=1}{\overset{r}{\sum}}2i\pi\mathbb{Z}P(e^{(k)})$  and \
dim$\left(\overline{\mathrm{g}_{u}}\right)=m$.
\end{enumerate}
\end{thm}
\medskip

This paper is organized as follows: In Section $2$,  we give
preliminary results. In Section $3$,  we prove the
Theorem~\ref{T:1}. A   parametrization  of an abelian subgroup of
$\mathcal{K}^{*}_{\eta,r}(\mathbb{C})$ is given in section 4.
    The proof of Theorems ~\ref{T:2}, ~\ref{T:3} and ~\ref{T:4} and Corollaries ~\ref{C:02} and ~\ref{C:2} are done in Section 5.
 Section 6 is devoted to give some examples.
\bigskip

\section{{\bf Preliminary \ results} }
\medskip

We present some results for abelian subgroup of $GL(n,
\mathbb{C})$.

\begin{prop}\label{p:1}$($\cite{WR}, Proposition 7'$)$ Let $A \in M_{n}(\mathbb{C})$. Then if no two eigenvalues of $A$ have
a difference of the form $2ik\pi$ , $k \in Z\backslash\{0\}$, then
$exp: \ M_{n}(\mathbb{C})\longrightarrow GL(n,\mathbb{C})$ is a
local diffeomorphism at $A$.
\end{prop}
\medskip

\begin{cor}\label{C:3} The restriction $exp/_{\mathcal{K}_{\eta,r}(\mathbb{C})} :\  \mathcal{K}_{\eta,r}(\mathbb{C})\longrightarrow\mathcal{K}^{*}_{\eta,r}(\mathbb{C})$
 is a local diffeomorphism.
\end{cor}
\bigskip

\begin{proof} The proof results from Proposition ~\ref{p:1} and the fact that
$exp_{/\mathcal{K}_{\eta,r}(\mathbb{C})}=exp_{/\mathbb{T}_{n_{1}}(\mathbb{C})}\oplus\dots\oplus
exp_{/\mathbb{T}_{n_{r}}(\mathbb{C})}$.
\end{proof}

\begin{prop}\label{p:2}$($\cite{mW},Theorem 2.1$)$ Let $H$ be a discrete additive subgroup of   $\mathbb{C}^{n}$.
Then there exist a basis $(u_{1},\dots,u_{n})$   of
$\mathbb{C}^{n}$  and  $1\leq r \leq n$  such that
$H=\underset{k=1}{\overset{r}{\sum}}\mathbb{Z}u_{k}$.
\end{prop}
\medskip

\begin{prop}\label{p:3}$($\cite{mW}, Theorem 3.1$)$ Let $F$ be a closed additive subgroup of   $\mathbb{C}^{n}$.
Then there exist a vector subspace $V$ of $\mathbb{C}^{n}$, over
$\mathbb{R}$ contained in $F$ and a vector subspace $W$ of
$\mathbb{C}^{n}$, over $\mathbb{R}$, such that:
\begin{itemize}
  \item[(i)]  $W\oplus V=\mathbb{C}^{n}$.
  \item [(ii)] $F\cap W$ is a discrete subgroup of  \  \ $\mathbb{C}^{n}$
  and $F=(F\cap W)\oplus V$.
 \end{itemize}
\end{prop}
\medskip

For any closed additive subgroup $F$ of $\mathbb{C}^{n}$, we have
\  dim$(F)=\mathrm{dim}(V)$.

\begin{cor}\label{C:4} $($Under above notations$)$ For every $y\in F\cap W$, there exist an open subset $O_{y}$ of  $\mathbb{C}^{n}$ \
such that $O_{y}\cap F=\{y\}+V$. In particular $O_{0}\cap F=V$.
\end{cor}
\smallskip

\begin{proof} We have $F=(F\cap W)\oplus V$,
where $F\cap W$ is a discrete subgroup of $W$. By Proposition
~\ref{p:2} there exists a basis $(u_{1}, \dots, u_{p})$ of $W$
and  $1\leq s \leq p$ such that $F\cap
W=\mathbb{Z}u_{1}\oplus\dots\oplus\mathbb{Z}u_{s}$. Let
$y=m_{1}u_{1}+  \ \dots \ + m_{s}u_{s}\in F\cap W$, take
$$O_{y}=\left]m_{1}-\frac{1}{2},\
m_{1}+\frac{1}{2}\right[u_{1}\oplus\dots\oplus\left]m_{p}-\frac{1}{2},
\ \frac{1}{2}+m_{p}\right[u_{p}\oplus V.$$ It follows  that
$O_{y}$ is an open subset of $\mathbb{C}^{n}$ such that $O_{y}\cap
F=V+\{y\}$. The proof is complete.
\end{proof}

\bigskip

\begin{prop}\label{p:4}  For every $u\in\mathbb{C}^{n}$,  there exist two vector
subspaces  $V$ and $W$  of   $\mathbb{C}^{n}$ over $\mathbb{R}$,
such that $V$ is contained in  $\overline{\mathrm{g}_{u}}$  and
 $W\oplus V=\mathbb{C}^{n}$, satisfying:

\begin{itemize}
  \item[i)]  $\overline{\mathrm{g}_{u}}=\left(\overline{\mathrm{g}_{u}}\cap W\right)\oplus V$,
with \ $\overline{\mathrm{g}_{u}}\cap W$ is discrete.
  \item [ii)] For every $\lambda\in \mathbb{R}$, we have
$\overline{\mathrm{g}_{\lambda
u}}=\left(\overline{\mathrm{g}_{\lambda u}}\cap W\right)\oplus V$,
with  $\overline{\mathrm{g}_{\lambda u}}\cap W$ is discrete
subgroup of $W$.
 \end{itemize}
\end{prop}
\smallskip

The proof uses the following lemma.

\begin{lem}\label{L:1}$($\cite{aAh-M05}, Lemmas 4.1 and  4.2$)$
\begin{itemize}
  \item [(i)] If $G$ is an abelian subgroup of $GL(n, \mathbb{C})$ then for every $u\in\mathbb{C}^{n}$, $\mathrm{g}_{u}$ is
an additive subgroup of $\mathbb{C}^{n}$.
  \item [(ii)] $exp(\mathrm{g})=G$.
  \end{itemize}
\end{lem}
\medskip

\begin{proof} [Proof of Proposition ~\ref{p:4}.] The proof of (i) results from Lemma ~\ref{L:1}.(i) and Proposition ~\ref{p:3}.\
\\
The proof of (ii)  follows from the fact that $\mathrm{g}_{\lambda
u}=\lambda \mathrm{g}_{u}$, $\lambda W=W$ and $\lambda V=V$, for
every $\lambda\in\mathbb{R}$.
\end{proof}
\smallskip

Let the fundamental result proved in \cite{aAh-M05}:

\begin{prop}\label{p:8}$($\cite{aAh-M05}, Proposition $2.3)$ Let  $G$ be an abelian
subgroup of  $GL(n, \mathbb{C})$. Then there exists $P\in GL(n,
\mathbb{C})$ such that\ $P^{-1}GP$  is an abelian subgroup of \
$\mathcal{K}^{*}_{\eta,r}(\mathbb{C})$, for some
$r\in\{1,\dots,n\}$ and $\eta\in(\mathbb{N}^{*})^{r}$.
\end{prop}
\bigskip

\section{{\bf Proof of Theorem~\ref{T:1}}}

 Throughout this section, suppose that $G$ is an abelian
subgroup of $\mathcal{K}^{*}_{\eta,r}(\mathbb{C})$ fore some
$r\in\mathbb{N}^{*}$ and
$\eta=(n_{1},\dots,n_{r})\in(\mathbb{N}^{*})^{r}$.
\
\\
Every $A\in G$ has the form $A=\mathrm{diag}(A_{1},\dots, A_{r})$
with $A_{k}\in\mathbb{T}_{n_{k}}(\mathbb{C})$ $k=1,\dots, r$.
Denote by:\
\\
- $G_{k}=\{A_{k},\ \ A\in G\}$, $k=1,\dots, r$.\
\\
-  $E(u_{k})=\mathrm{vect}(G_{k}(u_{k}))$, for every
$u=[u_{1},\dots,u_{r}]^{T}\in\mathbb{C}^{n}$.\
\\
- $\mathcal{G}=\mathrm{vect}(G)$,  the vector subspace of
$\mathcal{K}_{\eta,r}(\mathbb{C})$ generated by all elements of
$G$.
\\
-
$U=\underset{k=1}{\overset{r}{\prod}}\mathbb{C}^{*}\times\mathbb{C}^{n_{k}-1}$.
\medskip

\begin{lem}\label{L:5} Let $G$ be an abelian subgroup of $\mathcal{K}^{*}_{\eta,r}(\mathbb{C})$. Then for every $u\in\mathbb{C}^{n}$,
 $E(u)$ is $G$-invariant.
\end{lem}
\medskip

\begin{proof} Suppose that   $E(u)$ is generated by
$A_{1}u,\dots,A_{p}u$, with $A_{k}\in G$, $1\leq k\leq p$. Let
$w=\underset{k=1}{\overset{p}{\sum}}\alpha_{k}A_{k}u\in E(u)$ and
$B\in G$, then
$Bw=\underset{k=1}{\overset{p}{\sum}}\alpha_{k}BA_{k}u$. Since
$BA_{k}u\in G(u)\subset E(u)$  then it has the form
$BA_{k}u=\underset{j=1}{\overset{p}{\sum}}\beta_{k,j}A_{j}u$,
$\beta_{k,j}\in\mathbb{C}$, so $Bw=\underset{1\leq k,j\leq p
}{\sum}\alpha_{k}\beta_{k,j}A_{j}u\in E(u)$.
\end{proof}
\bigskip

\begin{lem}\label{L:6}$($\cite{S.Ch}, Proposition 3.1$)$ Let $G$ be an abelian subgroup of $GL(n,
\mathbb{C})$, $u\in \mathbb{C}^{n}$ and $v\in E(u)$.
\begin{itemize}
  \item [(i)] Then there
exist $B\in \mathcal{G}$ such that $Bu =v$.
  \item [(ii)] If \ $E(u) = E(v)$, \ then
$G(u)$ and $G(v)$ are isomorphic.
 \end{itemize}
\end{lem}
\bigskip

\begin{prop}\label{p:7+0+} Let $G$ be an abelian subgroup of $\mathcal{K}^{*}_{\eta,r}(\mathbb{C})$ and $u\in U$
such that $E(u)=\mathbb{C}^{n}$. Then  for every  $v\in U$ there
exists $B \in \mathcal{G}\cap GL(n, \mathbb{C})$ such that $Bu =
v$. In particular, $E(v)=\mathbb{C}^{n}$.
\end{prop}
\bigskip

\begin{lem}\label{L:7+0+} Let $G$ be an abelian subgroup of $\mathbb{T}^{*}_{n}(\mathbb{C})$ and
 $u\in \mathbb{C}^{*}\times\mathbb{C}^{n-1}$. Then
 for every  $v\in \mathbb{C}^{*}\times\mathbb{C}^{n-1}$ there
exist $B \in \mathcal{G}\cap GL(n, \mathbb{C})$ such that $Bu=v$.
\end{lem}
\bigskip

\begin{proof}  Let
$v\in \mathbb{C}^{*}\times\mathbb{C}^{n-1}$ since
$E(u)=\mathbb{C}^{n}$, by lemma ~\ref{L:6}, there exist $B \in
\mathcal{G}$ such that $Bu = v$. Since
$\mathbb{T}_{n}(\mathbb{C})$ is a vector space so
$\mathcal{G}\subset \mathbb{T}_{n}(\mathbb{C})$. Write
$u=[x_{1},\dots,x_{n}]^{T}$, $v=[y_{1},\dots,y_{n}]^{T}$ and
$$B=\left[\begin{array}{cccc}
  \mu_{B} & \ & \ & 0 \\
  a_{2,1} & \ddots & \ & \ \\
  \vdots & \ddots & \ddots & \ \\
  a_{n,1} & \dots & a_{n,n-1} & \mu_{B}
\end{array}\right],$$ then $\mu_{B}=\frac{y_{1}}{x_{1}}\neq 0$, hence $B\in GL(n,
\mathbb{C})$. As $\mathcal{G}\subset \mathcal{C}(G)$,  \
$B(E(u))=E(v)$.
\end{proof}
\bigskip

\begin{proof}[Proof of Proposition~\ref{p:7+0+}]Write  $u=[u_{1},\dots,u_{r}]^{T}$ and let $v=[v_{1},\dots,v_{r}]^{T}\in
U$,  with $u_{k},v_{k}\in\mathbb{C}^{n_{k}}$.  Since
$E(u)=\mathbb{C}^{n}$, then $E(u_{k})=\mathbb{C}^{n_{k}}$ for
every $k=1,\dots, r$. As $\mathcal{G}=\mathcal{G}_{1}\oplus\dots
\oplus \mathcal{G}_{r}$ and  $v_{k}\in \mathbb{C}^{*}\times
\mathbb{C}^{n_{k}-1}$, where $\mathcal{G}_{k}=Vect(G_{k})$. Hence
the proof results from Lemma~\ref{L:7+0+}. In particular,
$E(v)=B(E(u))=\mathbb{C}^{n}$.
\end{proof}
\medskip

\subsection{{\bf Proof of Theorem ~\ref{T:1}.}} Let $u\in \mathbb{C}^{n}$. By Proposition ~\ref{p:8}, we can assume that
$G\subset \mathcal{K}^{*}_{\eta,r}(\mathbb{C})$ and
$E(u)=\mathbb{C}^{n}$, otherwise, by Lemma~\ref{L:5}
 we replace $G$ by the restriction $G_{/E(u)}$. By construction, $\mathbb{C}^{n}\backslash U$ is union of $r$ $G$-invariant vector subspaces of
 $\mathbb{C}^{n}$ with dimension $n-1$, then we proves (iii) and we deduce that $u\in U$. Now, for every $v\in U$,
  there exists  by Proposition~\ref{p:7+0+}, $B \in \mathcal{G}\cap GL(n, \mathbb{C})$ such that $Bu= v$. Hence $E(v)=B(E(u))=\mathbb{C}^{n}$ and
  $G(v)=B(G(u))$, this proves (i) and (ii).  $\hfill{\Box}$
\bigskip

\section{{\bf Parametrization}}

Let $G$ be an abelian subgroup of
$\mathcal{K}^{*}_{\eta,r}(\mathbb{C})$ and $u\in \mathbb{C}^{n}$,
by Lemma ~\ref{L:5}, $E(u)$ is $G$-invariant, so consider the
linear map
\begin{align*}
\Phi_{u}:\  & vect\left(G_{/E(u)}\right)\longrightarrow  E(u)\\
 & A \mapsto Au
\end{align*}

\begin{prop}\label{p:41} For every $u\in \mathbb{C}^{n}\backslash\{0\}$, $\Phi_{u}$ is  a linear isomorphism.
\end{prop}
\medskip

\begin{proof} By construction,  $\Phi_{u}$ is surjective, since $\Phi_{u}\left(\textrm{Vect}\left(G_{/E(u)}\right)\right) = E(u)$.
\
\\
- $\Phi_{u}$ is injective: let $A\in \textrm{Ker}(\Phi_{u})$, so
$Au=0$. Let $x\in E(u)$, then by above there exists $B\in
\textrm{Vect}\left(G_{/E(u)}\right)$ such that $x=Bu$. As  $A\in
\textrm{Ker}(\Phi_{u})\subset vect\left(G_{/E(u)}\right)$
 then $AB = BA$. Therefore
$Ax=ABu=BAu=B(0)=0$. It follows that $A=0$ and hence
$\textrm{Ker}(\Phi_{u})= \{0\}$.
\end{proof}
\medskip

\begin{cor}\label{C:5} We have $\Phi^{-1}_{u}(G(u))=G_{/E(u)}$ and  $\Phi^{-1}_{u}(\mathrm{g}_{u})=\mathrm{g}_{/E(u)}$.
\end{cor}
\
\\
Under the above notation, we have:
\begin{prop}\label{p:15}  Let $G$ be an  abelian subgroup of
$\mathcal{K}^{*}_{\eta,r}(\mathbb{C})$. Then
$$exp(\Phi^{-1}_{u}(E(u)))\subset \Phi^{-1}_{u}(U_{u}).$$
\end{prop}
\
\\
To prove Proposition ~\ref{p:15}, we need the following Lemma:
\begin{lem}\label{L:10}  Let $G$ be an  abelian subgroup of
$\mathbb{T}_{n}^{*}(\mathbb{C})$. If $E(u)=\mathbb{C}^{n}$, then
$$exp(\Phi^{-1}_{u}(\mathbb{C}^{n}))\subset\Phi^{-1}_{u}(\mathbb{C}^{*}\times
\mathbb{C}^{n-1}).$$
\end{lem}
\bigskip

\begin{proof}  Here $G_{/E(u)}=G$ and $U_{u}=U=\mathbb{C}^{*}\times\mathbb{C}^{n-1}$.  First, one has $exp(\mathcal{G})\subset
\mathcal{G}$: indeed; for every $A\in \mathcal{G}$ we have
$A^{k}\in \mathcal{G}$, for every $k\in \mathbb{N}$ and so
$e^{A}=\underset{k\in\mathbb{N}}{\sum}\frac{A^{k}}{k!}\in
\mathcal{G}$. Moreover, we can check that $\mathcal{G}$ is the
subalgebra of $M_{n}(\mathbb{C})$ generated by $G$.
\
\\
Second, by corollary ~\ref{C:5}, $\Phi^{-1}_{u}(G(u))=G$. Since
$E(u)=\mathbb{C}^{n}$ and by Proposition~\ref{p:41}, $\Phi_{u}$ is
an isomorphism, then
\begin{align*}
 exp(\Phi^{-1}_{u}(\mathbb{C}^{n})) & =exp(\Phi^{-1}_{u}(E(u))) \\
  \ & =exp(\mathcal{G})\\
  \ & \subset \mathcal{G}=\Phi^{-1}_{u}(\mathbb{C}^{n})
\end{align*}

 Therefore
   $$exp(\Phi^{-1}_{u}(\mathbb{C}^{n}))  \subset \Phi^{-1}_{u}(\mathbb{C}^{n})\ \ \ \ (1)$$
\medskip

On the other hand,  by constraction of $\Phi_{u}$ and $U$, we have
$u\in U$ and so
$\Phi^{-1}_{u}(\mathbb{C}^{n})\cap\mathbb{T}^{*}_{n}(\mathbb{C})=\Phi^{-1}_{u}(\mathbb{C}^{*}\times\mathbb{C}^{n-1})$.
As  $exp(\Phi^{-1}_{u}(\mathbb{C}^{n}))\subset
\mathbb{T}^{*}_{n}(\mathbb{C})$ then by (1) we obtain
$$exp\left(\Phi^{-1}_{u}(\mathbb{C}^{n})\right)\subset
\Phi^{-1}_{u}(\mathbb{C}^{n})\cap\mathbb{T}^{*}_{n}(\mathbb{C})=\Phi^{-1}_{u}(\mathbb{C}^{*}\times\mathbb{C}^{n-1}).$$
\end{proof}
\bigskip

\begin{proof} [Proof of Proposition ~\ref{p:15}] Write $u=[u_{1},\dots,u_{r}]^{T}\in\mathbb{C}^{n}$. Suppose that
 $E(u)=\mathbb{C}^{n}$, otherwise we replace $G$ by $G_{/E(u)}$. So $G_{/E(u)}=G$ and $U_{u}=U$.  Since
$exp_{/\mathcal{K}_{\eta,r}(\mathbb{C})}=exp_{/\mathbb{T}_{n_{1}}(\mathbb{C})}\oplus\dots\oplus
exp_{/\mathbb{T}_{n_{r}}(\mathbb{C})}$
 and $\Phi^{-1}_{u}=\Phi^{-1}_{u_{1}}\oplus\dots\oplus\Phi^{-1}_{u_{r}}$. Then by Lemma ~\ref{L:10}
$$exp(\Phi^{-1}_{u}(U))=\underset{k=1}{\overset{r}{\prod}}exp_{/\mathbb{T}_{n_{k}}(\mathbb{C})}\left(\Phi^{-1}_{u_{k}}(\mathbb{C}^{n_{k}})\right)\subset
\underset{k=1}{\overset{r}{\prod}}\Phi^{-1}_{u_{k}}(\mathbb{C}^{*}\times\mathbb{C}^{n_{k}-1})=\Phi^{-1}(U).$$
\end{proof}

\bigskip
As consequence of Proposition ~\ref{p:15}, we have the following
results:
\begin{cor} \label{C:6} The map $f=\Phi_{u}\circ
exp_{/\Phi(\mathbb{C}^{n})}\circ\Phi_{u}^{-1} :
E(u)\longrightarrow U_{u}$ \ is well defined and continuous.
\end{cor}

\section{{\bf Proof of Theorems ~\ref{T:2}, ~\ref{T:3}, ~\ref{T:4} and Corollaries ~\ref{C:02} and ~\ref{C:2}}}

 By Proposition ~\ref{p:8}, suppose that $G$ is an abelian subgroup
of $\mathcal{K}^{*}_{\eta,r}(\mathbb{C})$ and then  \ $\mathrm{g}
= exp^{-1}(G)\cap \mathcal{K}_{\eta,r}(\mathbb{C})$.

In all this section fixed $u\in \mathbb{C}^{n}$ and suppose that
$E(u)=\mathbb{C}^{n}$ and $G_{/E(u)}=G$, leaving to replace $G$ by
$G_{/E(u)}$. Denote
 by $\Phi=\Phi^{-1}_{u}:\mathbb{C}^{n}\longrightarrow
 \mathcal{G}$. By Corollary~\ref{C:3},
 $exp:\mathcal{K}_{\eta,r}(\mathbb{C})\longrightarrow
 \mathcal{K}^{*}_{\eta,r}(\mathbb{C})$ is a local difeomorphism
 and by Proposition ~\ref{p:41}, $\Phi$ is an open map. Then we introduce the
 following Lemma which will be used in the proof of
 Theorem~\ref{T:2}.\
 \\
\begin{lem}\label{L:11} Let $O'$ be an open subset of $\mathbb{C}^{n}$ such that $exp_{/\Phi(O')}: \Phi(O')\longrightarrow exp(\Phi(O'))$ is a diffeomorphism. Then
$$exp\left(
\Phi(O')\right)\cap\overline{exp(\mathrm{g})}=exp\left(\Phi(O')\cap\overline{\mathrm{g}}\right).$$
\end{lem}
\bigskip

\begin{proof} Let $A\in exp(\Phi(O'))\cap
\overline{exp(\mathrm{g})}$, then $A=e^{\Phi(x)}$ for some $x\in
O'$. By Lemma ~\ref{L:1}.(ii), $exp(\mathrm{g})=G$, so $A\in G$.
Since $A\in \overline{exp(\mathrm{g})}$, there exists a sequence
$(B_{m})_{m\in \mathbb{N}}$ in $\mathrm{g}$ such that
$\underset{m\to +\infty}{lim}e^{B_{m}}=A=e^{\Phi(x)}$.  Since
$exp_{/\Phi(O')}:\Phi(O')\longrightarrow exp(\Phi(O'))$  is a
diffeomorphism, $exp(\Phi(O'))$ is an open set containing $e^{A}$,
so $e^{B_{m}}\in exp(\Phi(O'))$, $\forall m\geq p$,
 for some $p\in\mathbb{N}$. Then $B'_{m}=exp_{/\Phi(O')}^{-1}(e^{B_{m}})\in \Phi(O')$, $\forall m\geq p$. Since
 $e^{B_{m}}\in exp(\mathrm{g})=G$ and $\Phi(O')\subset\mathcal{K}_{\eta,r}(\mathbb{C})$,
 then  $B'_{m}\in exp^{-1}(G)\cap \mathcal{K}_{\eta,r}(\mathbb{C})=\mathrm{g}$, for every $m\geq p$.
 Therefore $\underset{m\to
+\infty}{lim}B'_{m}=\Phi(x)$, so $\Phi(x)\in \Phi(O')\cap
\overline{\mathrm{g}}$ and hence $A\in
exp\left(\Phi(O')\right)\cap \overline{\mathrm{g}}$.
\bigskip

Conversely, by continuity  of  $exp_{/\Phi(O')}:
\Phi(O')\longrightarrow exp(\Phi(O'))$, one has
$$exp\left(\Phi(O')\right)\cap \overline{\mathrm{g}}\subset
exp\left(\Phi(O')\right)\cap exp\left(
\overline{\mathrm{g}}\right)\subset exp\left(\Phi(O')\right)\cap
\overline{exp(\mathrm{g})}.$$ The prove is completed.
\end{proof}
\bigskip

\subsection{{\bf Proof of Theorem ~\ref{T:3} }}The equivalence $(i)\Longleftrightarrow (ii)$ follows from
Theorem ~\ref{T:1}.(i).\
\\
 $(i)\Longleftrightarrow (iii)$:  By Lemma ~\ref{L:5}, suppose that
 $E(u)=\mathbb{C}^{n}$and $U_{u}=U$ (leaving to replace $G$ by $G_{/E(u)}$). Then by Proposition~\ref{p:41} and Corollary ~\ref{C:5},
$\Phi=\Phi^{-1}_{u}:\mathbb{C}^{n}\longrightarrow \mathcal{G}$ is
an isomorphism satisfying $\Phi(G(u))=G$ and
$\Phi(\mathrm{g}_{u})=\mathrm{g}$.

By Proposition ~\ref{p:4}.(i), there exist a vector space $V$,
contained in $\overline{\mathrm{g}_{u}}$, a vector space $W$ such
that $V\oplus W=\mathbb{C}^{n}$ and
 $\overline{\mathrm{g}_{u}}=\left(\overline{\mathrm{g}_{u}}\cap
W\right)\oplus V,$ with $\overline{\mathrm{g}_{u}}\cap W$ is
discrete.
\bigskip

By corollary ~\ref{C:4} and Proposition~\ref{p:4}, there exists an
open subset $O$ of  $\mathbb{C}^{n}$  such that  $O\cap
\overline{\mathrm{g}_{u}}=V$. By Corollary ~\ref{C:3}, the
exponential map $exp \ : \ \mathcal{K}_{\eta,r}(\mathbb{C})\
\longrightarrow \ \mathcal{K}^{*}_{\eta,r}(\mathbb{C})$ is a
locally diffeomorphism, then there exists an open subset
$O'\subset O$, of $\mathbb{C}^{n}$ such that the restriction
$exp_{/\Phi(O')} \ : \ \Phi(O') \ \longrightarrow \
exp\left(\Phi(O')\right)$ of the exponential map on $\Phi(O')$ is
a diffeomorphism. Since $O'\subset O$, then $$O'\cap
\overline{\mathrm{g}_{u}}=O'\cap V.\ \ \ \ \ (1)$$

Since $O'\subset U_{u}$ \ then by Corollary ~\ref{C:6} the map
$f_{/O'}=\Phi^{-1}\circ exp_{/\Phi(O')}\circ\Phi:O'\longrightarrow
O'$ is well defined and as $exp_{/\Phi(O')}$ is a diffeomorphism
then $f_{/O'}$ is a diffeomorphism. By Lemma ~\ref{L:1}, one has
$exp(\mathrm{g})=G$ and then:

\begin{align*}
 f(O')\cap\overline{G(u)} &  =\Phi^{-1}\circ exp(\Phi(O'))\cap
\overline{G(u)} \\
  \ & =\Phi^{-1}\left(
exp\left(\Phi(O')\right)\cap \overline{\Phi(G(u))}\right)\\
 \ &  =\Phi^{-1}\left(
exp\left(\Phi(O')\right)\cap \overline{G}\right)\\
 \ & =\Phi^{-1}\left(
exp\left(\Phi(O')\right)\cap \overline{exp(\mathrm{g})}\right)
\end{align*}

By Lemma ~\ref{L:11} and by (1),  we obtain

\begin{align*}
   f(O')\cap\overline{G(u)} & =\Phi^{-1}\circ
exp\left(\Phi(O')\cap \overline{\mathrm{g}}\right) \\
  \ & =\Phi^{-1}\circ
exp\left(\Phi(O')\cap
\Phi\left(\overline{\mathrm{g}_{u}}\right)\right)\\
 \ & =\Phi^{-1}\circ
exp\circ\Phi\left(O'\cap \overline{\mathrm{g}_{u}}\right)\\ \ &
=\Phi^{-1}\circ exp\circ\Phi\left(O'\cap V\right)\\
 \ &  =f(O'\cap V).
\end{align*}

\medskip

As  $f(O')$ is an open subset of $\mathbb{C}^{n}$ and $O'\cap V$
is an open subset of the real vector space $V$, then   $f(O')\cap
\overline{G(u)}$  is a manifold with dimension
$m=\mathrm{dim}(V)=\mathrm{dim}(\overline{\mathrm{g}_{u}})$. We
conclude that $G(u)$ is regular with order $m$. Since
$\mathbb{C}^{n}$ is a real vector space with dimension $2n$ and
$V$ is a real subspace of $\mathbb{C}^{n}$ with dimension $m$, so
$m\leq 2n$. We conclude the equivalence $(i)\Longleftrightarrow
(iii)$. $\hfill{\Box}$

\medskip

\subsection{{\bf Proof  of Corollary ~\ref{C:02}}}\
\\
$\diamond$ {\it Complex Case: } Suppose that
$\mathbb{K}=\mathbb{C}$.\ Let $G$ be an abelian subgroup of $GL(n,
\mathbb{C})$ and  $u\in \mathbb{C}^{n}$. By Theorem ~\ref{T:3},
the orbit $G(u)$ is regular with order $m$ if and only if
$\mathrm{dim}(\overline{\mathrm{g}_{u}})=m$. Since
 $\overline{\mathrm{g}_{u}}$ is a closed additive subgroup of $\mathbb{C}^{n}$ (Lemma ~\ref{L:1}.(i)).
 Then $\mathrm{dim}(\overline{\mathrm{g}_{u}})\geq 0$, so $G(u)$ is regular
with order $m\geq 0$. This proves the complex case.\
\\
\\
$\diamond$ {\it Real  Case:} Suppose that
$\mathbb{K}=\mathbb{R}$.\ Let $G$ be an abelian subgroup of $GL(n,
\mathbb{R})$ and  $x\in \mathbb{R}^{n}$. So $G$ is considered as
an abelian subgroup of $GL(n, \mathbb{C})$. By the above case,
$G(x)$ is regular with some order $m$, with $m\leq 2n$. Then there
exists an open subset $O=O_{1}+iO_{2}$ of $\mathbb{C}^{n}$ with
$O_{1}, O_{2}$ are open subsets of $\mathbb{R}^{n}$, such that
$\overline{G(x)}\cap O$ is a manifold with dimension $m$. One has
$0\in O_{2}$, since $G(x)\subset \mathbb{R}^{n}$. Then
$\overline{G(x)}\cap O=\overline{G(x)}\cap O_{1}$ and $m\leq n$.
It follows that $G(x)$ is a regular orbit in $\mathbb{R}^{n}$ with
order $m$. The proof is completed. $\hfill{\Box}$
\
\bigskip

\subsection{{\bf Proof of Corollary ~\ref{C:2}}}\

\begin{lem}\label{L:12}$($\cite{aAhM05} Corollary 1.3$)$. {\it If  \ $G$  \ has a locally dense orbit $\gamma$ in
$\mathbb{C}^{n}$ then  $\gamma$ is dense in $\mathbb{C}^{n}$.}
\end{lem}
\medskip

 \begin{proof}[Proof of Corollary ~\ref{C:2}]\
\\
 (i)  If
 $\overline{G(u)}=\mathbb{R}^{n}$ then $\overline{G(u)}$ is a manifold with dimension $n$, so $G(u)$ is
 regular with order $m=n$.

 Conversely, if $G(u)$ is  regular with order $m=n$, then $\overline{G(u)}\cap O$ is a
 manifold with order $m=n$, for some open subset $O$ of $\mathbb{R}^{n}$. Hence $\overline{G(u)}\cap O$ is an open
 subset of $\mathbb{R}^{n}$. Therefore $G(u)$ is locally dense.
\\
\\
(ii) We use the same proof of (i) and by
 Lemma ~\ref{L:12} we have  $\overline{G(u)}=\mathbb{C}^{n}$.\
\end{proof}

\bigskip

\subsection{{\bf Proof of theorem ~\ref{T:4}}} If
$\overline{G(u)}$ is a vector subspace then
$\overline{G(u)}=E(u)$, so $G(u)$ is regular with order $2r(u)$.
\\
Conversely, if $G(u)$ is regular with order $2r(u)$, then
$\overline{G(u)}\cap O$ is a manifold with dimension $2r(u)$, for
some open set $O$. Since dim$(E(u))=2r(u)$ over $\mathbb{R}$, then
$\overline{G(u)}\cap O$ is an open subset of $E(u)$. So $G(u)$ is
locally dense in $E(u)$. By lemma ~\ref{L:12} applied on
$G_{/E(u)}$\ we have $\overline{G(u)}=E(u)$. $\hfill{\Box}$
\
\\
\\
\subsection{Algebric Lemmas}
\begin{lem}\label{L:13}$($\cite{mW}, Proposition 4.3$)$. Let $H = \mathbb{Z}u_{1}+.....+\mathbb{Z}u_{p}$
with $u_{k}=[u_{k,1},\dots, u_{k,n}]^{T}\in\mathbb{R}^{n}$, $k =
1,..., p$. Then $H$ is dense in $\mathbb{R}^{n}$ if and only if
for every $(s_{1},...,s_{p})\in
 \mathbb{Z}^{p}-\{0\}$ :
$$rank\left(\left[\begin{array}{cccc }
 u_{1,1} &\dots  &\dots  & u_{p,1} \\
 \vdots & \vdots & \vdots & \vdots \\
 u_{1,n} &\dots  &\dots  & u_{p,n} \\
  s_{1} &\dots &\dots  & s_{p }
 \end{array}\right]\right) =\ n+1.$$
 \end{lem}
\smallskip

\begin{cor}\label{C:7} Let $p\geq n+1$ and  $H = \mathbb{Z}u_{1}+\dots+\mathbb{Z}u_{p}$, $u_{k}\in \mathbb{R}^{n}$, $1\leq k \leq p$,
 such that  $(u_{1},\dots,u_{n})$ is a basis  of  $\mathbb{R}^{n}$.\
 If there exists $0\leq m\leq n$ such that $u_{k}=\underset{j=1}{\overset{m}{\sum}}\alpha_{k,j}u_{j+n-m}$,
for every $n+1\leq k \leq p$. Then the following assertions are
equivalent:\
\begin{enumerate}
\item [(i)] dim$(\overline{H})=m$.\
\item [(ii)] $u_{1},\dots,u_{p}$ satisfies property $\mathcal{D}(m)$.
\end{enumerate}
 \end{cor}
\bigskip

\begin{proof} Let $E=\mathbb{R}u_{n-m+1}\oplus\dots \oplus\mathbb{R}u_{n}$. We replace $\mathbb{R}^{n}$ by $E$
in Lemma~\ref{L:13} and  we obtain: $u_{1},\dots,u_{p}$ satisfies
property $\mathcal{D}(m)$ if and only if
$K=\mathbb{Z}u_{n-m+1}+.....+\mathbb{Z}u_{p}$ is dense in $E$ and
this is equivalent to $dim(\overline{H})=m$ since
$\overline{H}=\mathbb{Z}u_{1}+\dots+\mathbb{Z}u_{n-m}+\overline{K}$.
\end{proof}
\bigskip

\subsection{Proof of Theorem ~\ref{T:4}}
Denote by`:\
\\
- $u_{0}=[e_{1,1},\dots,e_{r,1}]^{T}$ and
$e_{k,1}=[1,0,\dots,0]^{T}\in\mathbb{C}^{n_{k}}$, $k=1,\dots, r$.\
\\
- $v_{0}=Pu_{0}$, where $P\in GL(n, \mathbb{C})$ is defined in
Proposition ~\ref{p:41} so that $P^{-1}GP\subset
\mathcal{K}^{*}_{\eta,r}(\mathbb{C})$.\
\\

\begin{prop}\label{p:017}$($\cite{aAh-M05}, Theorem 1.3$)$ Let $G$ be an abelian subgroup of $GL(n,\mathbb{C})$ generated
by $A_{1},\dots,A_{p}$. Let $B_{1},\dots,B_{p}\in \mathrm{g}$ such
that $A_{k} = e^{B_{k}}$ , $k = 1,\dots, p$. Then \
$\mathrm{g}_{v_{0}} =
\underset{k=1}{\overset{p}{\sum}}\mathbb{Z}B_{k}v_{0} +
\underset{k=1}{\overset{r}{\sum}} 2i\pi \mathbb{Z}Pe^{(k)}.$
\end{prop}
\bigskip

By construction of $\mathcal{K}^{*}_{\eta,r}(\mathbb{C})$, remark
that for every $u\in U$ there exists $Q\in
\mathcal{K}^{*}_{\eta,r}(\mathbb{C})$ such that $Qu_{0}=u$. Then
as a consequence of Proposition ~\ref{p:017} we get the following
Corollary:
\\
\begin{cor}\label{C:17} Let $G$ be an abelian subgroup of $\mathcal{K}^{*}_{\eta,r}(\mathbb{C})$ generated
by $A_{1},\dots,A_{p}$. Let $B_{1},\dots,B_{p}\in \mathrm{g}$ such
that $A_{k} = e^{B_{k}}$ , $k = 1,\dots, p$. Then for every $u\in
U$,  $\mathrm{g}_{u} =
\underset{k=1}{\overset{p}{\sum}}\mathbb{Z}B_{k}u +
\underset{k=1}{\overset{r}{\sum}} 2i\pi \mathbb{Z}Qe^{(k)},$ where
$Q\in \mathcal{K}^{*}_{\eta,r}(\mathbb{C})$ such that $Qu_{0}=u$.
\end{cor}
\bigskip

\begin{proof}[Proof of Theorem~\ref{T:4}]The proof of Theorem ~\ref{T:4} results from Theorem ~\ref{T:2},
Corollary ~\ref{C:7} and Corollary ~\ref{C:17}.
\end{proof}
\bigskip

\section{{\bf  Examples}}

Let $\mathbb{D}_{n}(\mathbb{C}) =
\left\{A=\mathrm{diag}(a_{1},\dots,a_{n})\ : \ \ a_{k}\in
\mathbb{C}^{*},\ 1\leq k \leq n \ \right\}$\ and let $G$ be an
abelian subgroup of $\mathbb{D}_{n}(\mathbb{C})$. In this case we
have \  $\widetilde{G}=G$ is a subgroup of
$\mathcal{K}_{(1,\dots,1),n}(\mathbb{C})$\ \ and \ \ $r_{G}=n$.

\medskip

\begin{exe} Let $G$ be the group generated by $A_{k}=\mathrm{diag}(\lambda_{k,1}e^{i\alpha_{k,1
}},\dots,\lambda_{k,n}e^{i\alpha_{k,n }})$, $k=1,\dots,p$,  where
$\lambda_{k,j}\in \mathbb{R}^{*}_{+}$,\  $\alpha_{k,j}\in
\mathbb{R}$, $1\leq j \leq n$. Let $u=[x_{1},\dots, x_{n}]^{T}\in
\mathbb{C}^{n}$, then the following assertions are equivalent:
\begin{enumerate}
\item [(i)]  $G(u)$  is regular with order $m$.
\item [(ii)] \  $u_{k}=[(log\lambda_{k,1}+i\alpha_{k,1
})x_{1},\dots,(log\lambda_{k,n}+i\alpha_{k,n })x_{n}]^{T}$, $1\leq
k \leq p$ \ with $2i\pi e_{1},\dots,2i\pi e_{n}$, satisfies
$\mathcal{D}(m)$.
\end{enumerate}
\end{exe}

\begin{proof}We let  $B_{k}=\mathrm{diag}(log\lambda_{k,1}+i\alpha_{k,1
},\dots,log\lambda_{k,n}+i\alpha_{k,n })$,

 One has $e^{B_{k}} = A_{k}$ and $B_{k}\in\mathbb{D}_{n}(\mathbb{C})$, $1 \leq k \leq p$. Then $B_{k}\in \mathrm{g}$.
  The result follows then from Theorem  ~\ref{T:4}.
\end{proof}
\medskip

\bigskip

 \begin{exe} $($\cite{aAhM05}, Example 6.2$)$ Let $G$ be the subgroup of $GL(4,\mathbb{R})$ generated by
$$A=\left[\begin{array}{cccc}
  1 & 0 & 0 &0 \\
  0 &1 &0 & 0 \\
  0 & 0 & 1 & 0 \\
  1 & 0 & 0 & 1
\end{array}\right]\ \  \mathrm{and} \ \ \ B = \left[\begin{array}{cccc}
  1 & 0 & 0 &0 \\
  0 &1 &0 & 0 \\
  0 & 0 & 1 & 0 \\
  0 & 1 & 0 & 1
\end{array}\right].$$ Then:

\begin{itemize}
  \item [i)] if $u\in(\mathbb{Q}^{*})^{2}\times\mathbb{R}^{2}$, $G(u)$ is
  discrete. So $G(u)$ is regular with order 0.
  \item [ii)] if $u\in\mathbb{Q}^{*}\times(\mathbb{R}\backslash\mathbb{Q})\times\mathbb{R}^{2}$, $G(u)$ is dense in a straight
  line. So $G(u)$ is regular with order 1.
 \end{itemize}
 \end{exe}
 \bigskip

In this example, $G$ is considered as a subgroup of
$GL(4,\mathbb{C})$. We have $\widetilde{G}=G$ is a subgroup of
$\mathbb{T}^{*}_{4}(\mathbb{C})$ and
$G(e_{1})=\left\{[1,0,0,n+m]^{T},\ \ \ n,m\in\mathbb{Z}\right\}$,
then $E(e_{1})=\mathbb{C}e_{1}+\mathbb{C}e_{4}$ and
$U_{e_{1}}=\mathbb{C}^{*}e_{1}+\mathbb{C}e_{4}$. So $E(e_{1})\neq
\mathbb{C}^{n}$.
\\
Let $B_{1}=A_{1}-I_{4}$ and $B_{2}=A_{2}-I_{4}$. Since
$B_{1}^{2}=B_{2}^{2}=0$, so  $e^{B_{1}}=A_{1}$ and
$e^{B_{2}}=A_{2}$. By Theorem ~\ref{T:4}, \
$\mathrm{g}_{u}=\mathbb{Z}B_{1}u+\mathbb{Z}B_{2}u+2i\pi \mathbb{Z}
e_{1}$.
\
\\
For every $u=[x,y,z,t]^{T}\in\mathbb{R}^{*}\times \mathbb{R}^{3}$,
we have $\mathrm{g}_{u}= (\mathbb{Z}x+\mathbb{Z}y)e_{4}+2i\pi
\mathbb{Z}e_{1}$, then: \
 \\
 -if $\frac{y}{x}\notin \mathbb{Q}$,
then $\overline{\mathrm{g}_{u}}=\mathbb{R}e_{4}+\mathbb{Z}e_{1}$,
so dim$(\overline{\mathrm{g}_{u}})=1$. By Theorem ~\ref{T:4},
$G(u)$ is regular with order 1.\
\\
 -if $\frac{y}{x}\in \mathbb{Q}$,
then $\overline{\mathrm{g}_{u}}=\mathbb{Z}ae_{4}+\mathbb{Z}e_{1}$,
for some $a\in \mathbb{R}$, so dim$(\overline{\mathrm{g}_{u}})=0$.
By Theorem ~\ref{T:4}, $G(u)$ is regular with order 0.\
$\hfill{\Box}$
\bigskip

 A simple example for $n=1$ and $\mathbb{K}=\mathbb{R}$ is given in the following, to
 show that the closure of a regular orbit is not necessarily a
 manifold.

  \begin{exe} Let $\lambda>1$ and $G$ be the group generated by
  $\lambda.id_{\mathbb{R}}$, then for every $x\in\mathbb{R}^{*}$,
  we have $G(x)=\{\lambda^{n}, \ n\in\mathbb{Z}\}$ and  $\overline{G(x)}=\{\lambda^{n}, \
  n\in\mathbb{Z}\}\cup\{0\}$. Thus $G(x)$ is discrete, so it is
  regular with order 0, but $\overline{G(x)}$ is not a manifold.
  \end{exe}
\bigskip

\bibliographystyle{amsplain}
\vskip 0,4 cm

\end{document}